\documentclass[11pt,a4paper]{article}

\usepackage[utf8]{inputenc}  
\usepackage[T1]{fontenc}  
\usepackage[english]{babel}
\usepackage{mathtools}
\usepackage{tikz}
\usepackage{bm}

\usetikzlibrary{positioning,chains,fit,shapes,calc}

\usepackage{geometry}
\usepackage{todonotes}

\usepackage{enumitem}
\usepackage{amsmath}
\usepackage{amsfonts}
\usepackage{amssymb}
\usepackage{amsthm}
\usepackage{stmaryrd}
\SetSymbolFont{stmry}{bold}{U}{stmry}{m}{n}
\usepackage{bbm}

\newtheorem{theorem}{Theorem}
\newtheorem{lemma}[theorem]{Lemma}

\newtheorem{proposition}[theorem]{Proposition}
\newtheorem{remark}[theorem]{Remark}
\newtheorem{cor}[theorem]{Corollary}
\newtheorem{definition}[theorem]{Definition}
\newtheorem*{theorem*}{Theorem}
\newtheorem*{examples}{Examples}
\newtheorem*{proposition*}{Proposition}
\newtheorem*{lemma*}{Lemma}
\usepackage{graphicx}
\usepackage{float}
\usepackage{cases}
\restylefloat{figure}

\renewcommand{\P}{\mathbb{P}}

\newcommand{\R}{\mathbb{R}}
\newcommand{\N}{\mathbb{N}}
\newcommand{\Z}{\mathbb{Z}}
\newcommand{\Q}{\mathbb{Q}}

\newcommand{\cT}{\mathcal{T}}

\newcommand{\be}{\mathbf{e}}
\newcommand{\cC}{\mathcal{C}}

\newcommand{\cW}{\mathcal{W}}
\newcommand{\cM}{\mathcal{M}}

\newcommand{\cU}{\mathcal{U}}
\newcommand{\bT}{\mathbb{T}}
\newcommand{\bk}{\mathbf{k}}
\newcommand{\br}{\mathbf{r}}
\newcommand{\bg}{\mathbf{g}}

\newcommand{\bN}{\mathbf{N}}
\newcommand{\bb}{\mathbf{b}}
\newcommand{\bc}{\mathbf{c}}
\newcommand{\bw}{\mathbf{w}}
\newcommand{\bx}{\mathbf{x}}
\newcommand{\by}{\mathbf{y}}
\newcommand{\bz}{\mathbf{z}}
\newcommand{\hz}{\hat{\zeta}}
\newcommand{\bhz}{\bm{\hat{\zeta}}}
\newcommand{\ga}{\bm{\gamma}}
\newcommand{\bmu}{\bm{\mu}}
\newcommand{\tbmu}{\bm{\tilde{\mu}}}
\newcommand{\bze}{\bm{\zeta}}
\newcommand{\tbze}{\bm{\tilde{\zeta}}}
\newcommand{\tmu}{\tilde{\mu}}

\newcommand{\phii}{\phi^{(i)}}
\newcommand{\tphii}{\tilde{\phi}^{(i)}}
\newcommand{\tmui}{\tilde{\mu}^{(i)}}
\newcommand{\mui}{\mu^{(i)}}

\usepackage{MnSymbol,wasysym}
\usepackage[colorlinks=false]{hyperref}

\author{Paul Thévenin}
\title{Critical exponential tiltings for size-conditioned multitype Bienaymé--Galton--Watson trees.}
\date{}

\begin{document}

\maketitle

\begin{abstract}
\small{We consider here multitype Bienaymé--Galton--Watson trees, under the conditioning that the numbers of vertices of given type satisfy some linear relations. We prove that, under some smoothness conditions on the offspring distribution $\bze$, there exists a critical offspring distribution $\tbze$ such that the trees with offspring distribution $\bze$ and $\tbze$ have the same law under our conditioning. This allows us in a second time to characterize the local limit of such trees, as their size goes to infinity. Our main tool is a notion of exponential tilting for multitype Bienaymé--Galton--Watson trees.}
\end{abstract}

\section{Introduction}

The main purpose of this paper is to study the asymptotic behaviour of multitype Bienaymé--Galton--Watson trees (or BGW trees), which are a famous model of random trees used initially to describe the evolution of a population. Roughly speaking, vertices of the tree are individuals who have children independently according to a given distribution. In addition, each individual is given a type (which is for us an integer). We consider here only the case where the number $K$ of possible types is finite. A question of interest is to take such a BGW tree $\cT_n$ conditioned to have size $n$ (where the size of a tree is a parameter that needs to be defined), and investigate the asymptotic properties of the tree $\cT_n$ when $n \rightarrow \infty$. In the monotype case (i.e. when $K=1$), the natural notion of size is the total number of vertices and the question has been extensively investigated. The first results on the structure of $\cT_n$ for $n$ large date back to Kesten \cite{Kes86} (see also Janson \cite{Jan12}), who proves the local convergence of $\cT_n$. In words, balls of fixed radius around the root of the tree converge in distribution. The limiting object, the so-called Kesten tree, is made of an infinite spine on which i.i.d. subtrees are grafted. In another direction, still in the monotype case, Aldous \cite{Ald91a, Ald91b, Ald93} shows the convergence of the tree $\cT_n$ seen as a metric space, where edges of the tree are rescaled to have length $n^{-1/2}$, to a limiting random metric space called Aldous' Brownian Continuum Random Tree.

Analogous results exist when $2 \leq K < \infty$. In this multitype setting, different definitions of the size of a tree are possible: total number of vertices, number of vertices of type $1$, among others. Under diverse assumptions, Pénisson \cite{Pen14}, Abraham-Delmas-Guo \cite{ADG18} or Stephenson \cite{Ste18} characterize the local limit of multitype BGW trees. On the other hand, Miermont \cite{Mie08} and more recently Haas and Stephenson \cite{HS21} prove the convergence of multitype BGW trees, under an assumption of finite covariance, towards Aldous' Continuum Random tree. In all these results, an important assumption made on the tree is that the distribution of the offspring of a vertex is critical (in the case $K=1$, this corresponds to the fact that the average number of children of an individual is $1$).

Again, in the monotype case, such results are known. Janson \cite{Jan12} shows that, when $K=1$ and an offspring distribution $\mu$ is given, it is possible to characterize offspring distributions $\tilde{\mu}$ with the following property: for all $n$, let $\cT_n$ be a $\mu$-BGW tree and $\tilde{\cT}_n$ a $\tilde{\mu}$-BGW tree, conditioned to have $n$ vertices. Then, $\cT_n$ and $\tilde{\cT}_n$ have the same distribution. These distributions $\tilde{\mu}$ are obtained from $\mu$ by performing an operation called exponential tilting. If there exists such a $\tilde{\mu}$ which is critical, then local and scaling limit results that hold for $\tilde{\cT}_n$ also hold for $\cT_n$. Janson \cite{Jan12} covers also cases where such a $\tilde{\mu}$ does not exist, and where condensation phenomena may appear.
Our aim here is to extend the scope of these results, by generalizing the notion of exponential tilting to the multitype case.

\paragraph*{Acknowledgements}

The author would like to thank Svante Janson and Stephan Wagner for insightful discussions, comments and corrections. The author acknowledges the support of the Austrian Science Fund (FWF) under grant P33083.

\paragraph{General notation}

In the whole paper, we let $\N := \{ 0, 1, 2 \ldots \}$ be the set of nonnegative integers and $\N^* := \{ 1, 2, \ldots \}$ be the set of positive integers. For $K \in \N^*$, we set $[K]=\{1, \ldots, K\}$. Furthermore, we will write $\bm{0}$ for $(0, \ldots, 0) \in \R^d$ for a given $d$ (the value of $d$ will always be made clear by the context).

\section{Background on trees}
\label{sec:trees}

We start by recalling some definitions and useful well-known results concerning BGW trees.

\subsection{Plane trees.} We first define plane trees using Neveu's formalism \cite{Nev86}. We let $\cU := \bigcup_{k \geq 0} (\N^*)^k$ be the set of finite sequences of positive integers, with the convention that $(\N^*)^0=\{\varnothing\}$. By a slight abuse of notation, for $k \in \N$, we write an element $u$ of $(\N^*)^k$ by $u=u_1 \cdots u_k$, with $u_1, \ldots, u_k \in \N^*$. For $k \in \N$, $u=u_1\cdots u_k \in (\N^*)^k$ and $i \in \N$, we denote by $ui$ the element $u_1 \cdots u_ki \in (\N^*)^{k+1}$ and by $iu$ the element $iu_1 \cdots u_k \in (\N^*)^{k+1}$. A plane tree $t$ is a subset of $\cU$ satisfying the following three conditions:
(i) $\emptyset \in t$ (the tree has a root); (ii) if $u=u_1\cdots u_n \in t$, then, for all $k \leq n$, $u_1\cdots u_k \in t$ (these elements are called ancestors of $u$); (iii) for any $u \in t$, there exists a nonnegative integer $k_u(t)$ such that, for every $i \in \N^*$, $ui \in t$ if and only if $1 \leq i \leq k_u(t)$ ($k_u(t)$ will be called the number of children of $u$, or the outdegree of $u$, and an element of the form $ui$ is called a child of $u$). The elements of $t$ are called vertices, and we denote by $|t|$ the total number of vertices of $t$. Finally, we denote by $\bT$ the set of plane trees.

\paragraph*{Multitype plane trees}

Fix $K \in \N^*$ and let $[K] := \{ 1, \ldots, K \}$ be the set of types. A $K$-type plane tree is a pair $T := (t,\be_t)$ where $t \in \bT$ is a plane tree and $\be_t: t \mapsto [K]$ is a map associating a type with each vertex of $t$. For $u \in t$, $\be_t(u)$ is called the type of the vertex $u$. For all $ \in [K]$, we also denote by $N_i(T)$ the number of vertices $u$ of the tree $t$ such that $\be_t(u)=i$. We let $\bT^{(K)}$ be the set of $K$-type plane trees and, for $i \in [K]$, we denote by $\bT^{(K,i)}$ the subset of $\bT^{(K)}$ of trees whose root has label $\be_t(\emptyset)=i$.

\subsection{Multitype BGW trees}

We now define our main model of random trees, which we call $K$-type BGW trees. For $K \in \N^*$, set $\cW_K := \bigcup_{n \geq 0} [K]^n$. Let $\bze := (\zeta^{(i)})_{i \in [K]}$ be a family of probability distributions on $\cW_K$. Let $(X_u^i, u \in \cU, i \in [K])$ be a family of independent variables with values in $\cW_K$ such that, for all $(u,i) \in \cU \times [K]$, $X_u^i$ is distributed according to $\zeta^{(i)}$. We also denote by $|X_u^i|$ the size of the vector $X_u^i$. Now fix $i \in [K]$. We recursively construct a (random) $K$-type tree $\cT^{(i)} := (t, \be_t)$ with values in $\bT^{(K,i)}$, as follows:
\begin{itemize}
\item $\varnothing \in t, \be_t(\varnothing) = i$;
\item if $u \in t$ and $\be_t(u) = j$, then, for $k \in \N^*$, $uk \in t$ if and only if $1\leq k \leq |X_u^j|$ and in this case $\be_t(uk)=X_u^j(k)$.
\end{itemize}

In other words, the root of $\cT^{(i)}$ has type $i$ and vertices of type $j$ in $\cT^{(i)}$ have children independently according to $\zeta^{(j)}$. We call $\cT^{(i)}$ a $\bze$-BGW tree. Note that $\cT^{(i)}$ may be finite or infinite.

It is useful in our context to define the \textit{projection} of the family $\bze$. First, for any $w \in \cW_K$ and $j \in [K]$, let $w^{(j)}$ be the number of $j$'s in $w$. Define the projection of $w$ as the element $p(w) = (w^{(1)}, \ldots, w^{(K)}) \in \N^K$. For $i \in [K]$, denote by $\mu^{(i)}$ the probability distribution on $\N^K$ defined by: for all $(k_1, \ldots, k_K) \in \N^K$,
\begin{align*}
\mu^{(i)}(k_1, \ldots, k_K) = \sum_{\substack{w \in \cW_K \\ p(w)=(k_1, \ldots, k_K)}} \zeta^{(i)}(w). 
\end{align*}

It turns out that numerous asymptotic structural properties of $\cT^{(i)}$ only depend on the projection $\bmu:=(\mu^{(i)}, i \in [K])$. 
In this paper, we will only consider nondegenerate $\bze$, that is, such that its projection $\bmu$ satisfies:
\begin{align*}
\exists i \in [K], \mui\left( \left\{ \bz, \sum_{j \in [K]} z_j \neq 1 \right\} \right) > 0.
\end{align*}

We define the mean matrix of $\bze$, $M := (m_{i,j})_{i,j \in [K]}$ as the $K \times K$ matrix such that
\begin{align*}
m_{i,j} = \sum_{\bz \in \N^K} z_j \mu^{(i)}(\bz).
\end{align*}
In other words, $m_{i,j}$ is the expected number of children of type $j$ of a vertex of type $i$.

We say that $\bmu$ is entire if, for all $i$, the generating function $\phii$ of $\mui$ is entire, and we say that $\bze$ is entire if its projection $\bmu$ is entire. We say that $\bze$ is critical (by convention, we will also say that its projection $\bmu$ is critical) if the spectral radius $\rho(M)$ of $M$ is equal to $1$. We say that $\bze$ is irreducible (again, we also say that $\bmu$ is irreducible) if, for all $i,j \in [K]$, there exists $p \in \N^*$ such that $M^p_{i,j} > 0$. In particular, all these properties of $\bze$ only depend on its projection $\bmu$.

\subsection{Conditioning a $K$-type tree}

\paragraph*{History and results}

The asymptotic structure of large multitype BGW trees has been a topic of interest in the past few years. People have in particular studied the so-called scaling limit of such trees: seeing a tree as a metric space, does $\cT^{(i)}$, conditioned to have a large size, converge after renormalization as a metric space? In the monotype case, the notion of size is usually the number of vertices in the tree. Under mild conditions, Aldous \cite{Ald91a} shows that a $\zeta$-BGW tree conditioned to have $n$ vertices converges, after rescaling distances by $\sqrt{n}$, to a limiting object called Aldous' Brownian Continuum Random Tree (or CRT). In the multitype case, there are many possible notions of size, and thus many possible conditionings: by the total number of vertices, by the number of vertices of a given type or by the numbers of vertices of each type for example. Haas-Stephenson \cite{HS21} (see also Miermont \cite{Mie08} for a slightly weaker result) proves that, under a finite covariance assumption, a $K$-type BGW tree $\cT_n$ conditioned to have $n$ vertices of type $1$ converges after renormalization towards the Brownian CRT. One of their crucial hypotheses is that the offspring distribution $\bze$ that they consider must be critical.

On the other hand, a lot of attention has been given to the so-called local limit of multitype BGW trees. We say that the tree $\cT_*$ is the local limit of a sequence $(\cT_n)$ of trees if, for any fixed $r \geq 0$, the ball of radius $r$ centered at the root of $\cT_n$, seen as a random rooted plane tree, converges in distribution towards the ball of radius $r$ centered at the root of $\cT_*$. In the monotype case, Kesten \cite{Kes86} first introduced a discrete infinite tree, Kesten's tree, which is the local limit of size-conditioned $\mu$-BGW trees, where $\mu$ is any critical offspring distribution (see Janson \cite{Jan12} for a proof). Recently, some multitype generalizations have been proven, under different conditionings. A vector $(a_1, \ldots, a_K) \in [0,1]^K$ of sum $1$ being given, Pénisson \cite{Pen14} (see also Abraham-Delmas-Guo \cite{ADG18}) proves under some smoothness condition the local convergence of the $K$-type tree $\cT_{\bk(n)}$ conditioned to have $k_i(n)$ vertices of type $i$, where $\bk(n):=(k_i(n), i \in [K])$ is a sequence of vectors such that, for all $i$,
\begin{align*}
\lim_{n \rightarrow \infty} \frac{k_i(n)}{\sum_{j=1}^K k_j(n)} = a_i.
\end{align*}
Stephenson \cite{Ste18} shows, under an assumption of exponential
moments, the local convergence of a critical multitype BGW tree conditioned on a linear combination of its type population. Again, their main assumption is that the offspring distribution $\bze$ is critical. See Section \ref{sec:convergenceoftrees} for more details.

The main goal of this paper is to obtain such results without the criticality assumption: a non-critical distribution $\bze$ being given, does a $\bze$-BGW tree (under some conditioning) admit a scaling limit or a local limit? In the monotype case, it turns out (see Janson \cite[Section $4$]{Jan12}) that, under very mild conditions on a distribution $\bze$, a size-conditioned $\bze$-BGW tree converges after renormalization to the Brownian CRT, and locally to Kesten's tree. In particular, it is the case when $\bze$ is entire, or when $\bze$ is supercritical.
In the multitype case, Pénisson \cite[Lemma $5.3$]{Pen14} shows that, under a smoothness condition, a $\bze$-BGW tree conditioned to have $k_i$ vertices of type $i$ for all $i \in [K]$ is distributed as a $\tbze$-BGW tree with the same conditioning, for some critical offspring distribution $\tbze$. Our goal is to extend this result, which would allow to prove limit results for possibly non-critical multitype BGW trees.

\paragraph*{General conditionings}

We consider a fairly large class of conditionings. Fix $L \geq 1$ and let $\Gamma \in \mathcal{M}_{L,K}(\R)$. Fix $i \in [K]$, a $L$-tuple $\bg := (g_1, \ldots, g_L) \in \R^L$. We consider the tree $\cT^{(i)}_{\Gamma,\bg}$, which is the $\bze$-BGW tree $\cT^{(i)}$ under the conditioning
\begin{equation}
\label{eq:gamma}
\Gamma \begin{pmatrix}
N_1(\cT^{(i)})\\
\vdots \\
N_K(\cT^{(i)})
\end{pmatrix}
= 
\begin{pmatrix}
g_1\\
\vdots \\
g_L
\end{pmatrix}.
\end{equation}

Observe that
\begin{itemize}
\item[(i)] if $\Gamma = \begin{pmatrix}
\gamma_1 & \cdots & \gamma_K 
\end{pmatrix} \in \cM_{1,K}(\N)$, we are in Stephenson's case;

\item[(ii)] if $\Gamma$ is the identity matrix $\in \cM_{K,K}(\Z)$, we are in Pénisson and Abraham-Delmas-Guo's case;
\item[(iii)] if $\Gamma = \begin{pmatrix}
1 & 0 & \cdots & 0
\end{pmatrix} \in \cM_{1,K}(\Z)$, we are in Miermont and Haas-Stephenson's case.
\end{itemize}

\begin{definition}
\label{def:gammaequiv}
Fix $L \geq 1$ and $\Gamma \in \cM_{L,K}(\R)$. We say that two families $\bze, \tbze$ are $\Gamma$-equivalent if the following two conditions hold:
\begin{itemize}
\item[(i)] for all $\bx \in \cW_K$, all $i \in [K]$, $\zeta^{(i)}(\bx) = 0$ if and only if $\tilde{\zeta}^{(i)}(\bx) = 0$;
\item[(ii)] for all $i \in [K]$, all $\bg$ such that \eqref{eq:gamma} holds for a $\bze$-BGW with positive probability, we have in distribution
\begin{align*}
\cT_{\Gamma,\bg}^{(i)} \overset{(d)}{=} \tilde{\cT}^{(i)}_{\Gamma,\bg},
\end{align*}
where $\tilde{\cT}^{(i)}_{\Gamma,\bg}$ is a $\tbze$-BGW tree with root label $i$ conditioned on \eqref{eq:gamma}.
\end{itemize}
\end{definition}

Observe that, under these assumptions, irreducibility of $\bze$ implies irreducibility of $\tbze$. It is clearly an equivalence relation. 
\begin{remark}
\label{rk:equiv}
In the monotype case, it turns out that we can obtain (i) as a consequence of (ii), and thus we only need Assumption (ii). However, in the multitype case, it may happen that (ii) does not imply (i), and thus that assuming only (ii) does not define an equivalence relation. For example, consider the case $K=2$, the matrix $\Gamma := I_2$ and the two distributions $\bze, \tbze$ defined as follows:

\begin{itemize}
\item $\zeta^{(1)}(\emptyset) = \zeta^{(1)}(1,2) = 1/2, \zeta^{(2)}(\emptyset) = \zeta^{(2)}(1,2) = 1/2$;
\item $\tilde{\zeta}^{(1)}(\emptyset) = \tilde{\zeta}^{(1)}(1,2) = \tilde{\zeta}^{(1)}(1,1,1,1,2) = 1/3, \tilde{\zeta}^{(2)}(\emptyset) = \tilde{\zeta}^{(2)}(1,2) = \tilde{\zeta}^{(2)}(1,1,1,1,2) = 1/3$.
\end{itemize}

In this case, (ii) holds for $\bze$ but not for $\tbze$. Indeed, for any $n \geq 1$, we have that $$\P\left( N_2(\cT^{(1)})=n-1|N_1(\cT^{(1)})=n\right)=1,$$ while $$\P\left( N_2(\tilde{\cT}^{(1)})=n-1|N_1(\tilde{\cT}^{(1)})=n\right) \in (0,1).$$

We conjecture however that, if $\Gamma \in \cM_{1,K}(\Z_+)$, then (ii) implies (i).
\end{remark}

We can now state our main theorem. To this end, the technical conditions that we will assume on the offspring distribution $\bze$ are the following: 
\begin{enumerate}[label=(\textbf{A.\arabic*})]
\item $\bze$ is entire.
\label{cond:entire}
\item  For all $j \in [K]$, $\zeta^{(j)}(\emptyset) > 0$. \label{cond:jamaiszero}

\item for all $i \in [K]$, for $b_i$ large enough, uniformly in $(b_j)_{j \neq i} \in \R_+^{K-1}$, we have $\frac{\partial \phi^{(i)}(b_1, \ldots, b_K)}{\partial b_i} \geq \phi^{(i)}(b_1, \ldots, b_K)/b_i$.  \label{cond:escape}
\end{enumerate}

Here, recall that $\phi^{(i)}$ denotes the generating function of $\mu^{(i)}$, where $\bmu := (\mu^{(1)}, \ldots, \mu^{(K)})$ is the projection of $\bze$.
We will also only consider matrices $\Gamma$ satisfying the following condition:

\begin{enumerate}[label=(\textbf{B})]
\item There exists $\ga \in (Ker \Gamma)^\perp$ such that $\ga \in (\N^*)^K$ and $\gamma_i=1$ for some $i \in [K]$. \label{condition}
\end{enumerate}

\begin{remark}
Observe that \ref{cond:entire}-\ref{cond:escape} are technical smoothness conditions on the distribution $\bze$, while \ref{condition} is only a condition on the matrix $\Gamma$. It is not clear that these conditions can be easily lifted, see Section \ref{sec:questions}.
\end{remark}

\begin{examples}
An interesting example is when there exist $f_1, \ldots, f_K$ entire functions with nonnegative coefficients such that, for all $i \in [K]$, $\phi^{(i)} = e^{f_i-f_i(1,\ldots,1)}$ and $f_i(0,\ldots, 0, b_i, 0, \ldots, 0) \rightarrow \infty$ as $b_i \rightarrow \infty$. In this case, it is clear that Assumptions \ref{cond:entire}-\ref{cond:escape} are satisfied. This includes, for example, exponentials of polynomials.
\end{examples}

We can now expose our main theorem, which states the existence of critical $\Gamma$-equivalent distributions under Assumptions \ref{cond:entire}-\ref{cond:escape} and \ref{condition}.

\begin{theorem}
\label{thm:main}
Let $\bze$ be a probability distribution satisfying \ref{cond:entire}-\ref{cond:escape} and a matrix $\Gamma$ such that \ref{condition} holds. Then, there exists a critical distribution $\tbze$ that is $\Gamma$-equivalent to $\bze$. Furthermore, if $rk(\Gamma)=1$ and $\bze$ is irreducible, then this critical distribution is unique.
\end{theorem}

It is worth noticing that, in the monotype case, only Assumption \ref{cond:entire} is needed, as \ref{cond:escape}, \ref{cond:jamaiszero} and \ref{condition} come for free. However, the proof (see \cite{Jan12}) makes use of a continuity argument which is not valid anymore with two or more types.

The main idea to prove Theorem \ref{thm:main} is to introduce a family of multitype exponential tiltings, generalizing the results of \cite{Jan12}. The assumptions made on $\bze$ ensure the existence of a critical exponential tilting of $\bze$ with is $\Gamma$-equivalent to $\bze$. In particular, the following holds.

\begin{cor}
\label{cor:loclim}
Let $\bze$ be an irreducible distribution satisfying \ref{cond:entire}-\ref{cond:escape} and $\Gamma$ satisfying \ref{condition}. Fix $j \in [K]$, and let $(k_n, n \geq 1)$ be a sequence of positive integers such that $k_n \rightarrow \infty$ and, for all $n$, 
\begin{align*}
\P\left( \sum_{i=1}^K \gamma_i N_i(\cT^{(j)}) = k_n \right) > 0.
\end{align*} 
Then, there exists a discrete infinite $K$-type tree $\cT_*$ such that
\begin{align*}
\cT^{(j)}_{\Gamma, k_n} \underset{n \rightarrow \infty}{\overset{(loc)}{\rightarrow}} \cT_*,
\end{align*}
where $\cT^{(j)}_{\Gamma, k_n}$ is the tree $\cT^{(j)}$ conditioned on $\sum_{i=1}^K \gamma_i N_i(\cT^{(j)}) = k_n$.
\end{cor}

\paragraph{Overview of the paper}

We start by defining the notion of multitype exponential tiltings and describe a class of $\Gamma$-equivalent distributions in Section \ref{sec:tiltings}. Then, Section \ref{sec:proofs} is devoted to the proof of our main result, Theorem \ref{thm:main}, and Section \ref{sec:convergenceoftrees} to the proof of Corollary \ref{cor:loclim}, concerning local limits of non-critical multitype trees. In the last section, Section \ref{sec:questions}, we list a few open questions, mainly on the possibility of lifting our different assumptions \ref{cond:entire}-\ref{cond:escape} and \ref{condition}.

\section{Exponential tiltings}
\label{sec:tiltings}

In this section, we provide a sufficient criterion for two distributions to be $\Gamma$-equivalent. Observe that, the same way as we define equivalent families of distributions on $\cW_K$, we can define equivalent families of distributions on $\N^K$ as follows.

We denote in what follows $\Gamma \in \cM^{(K)}(\R) := \bigcup_{L \geq 1} \cM_{L,K}(\R)$.

\begin{definition}
Let $\Gamma \in \cM^{(K)}(\R)$.
We say that two families $\bmu$, $\tbmu$ odf distributions on $\N^K$ are $\Gamma$-equivalent if there exist two families $\bze, \tbze$ of distributions on $\cW_K$ such that $\bmu$ is the projection of $\bze$, $\tbmu$ is the projection of $\tbze$, and $\bze$ and $\tbze$ are $\Gamma$-equivalent. 
\end{definition}

\begin{proposition}
\label{prop:gammaequivalence}
The $\Gamma$-equivalence on distributions on $\N^K$ is an equivalence relation.
\end{proposition}

\begin{proof}
It is clear that any distribution on $\N^K$ is the projection of a distribution on $\cW_K$. More precisely, a distribution $\bmu$ on $\N^K$ being given, we can characterize the distributions $\bze$ on $\cW_K$ whose projection is $\bmu$: $\bze$ has projection $\bmu$ if and only if there exist probability measures $(\nu_{i,e}, i \in [K], e \in \N^K)$ on $\cW_K$ indexed by $[K] \times \N^K$, such that for all $e := (e_1, \ldots, e_K)$, $\nu_{i,e}$ takes its values in $\cW^{(e)}_K := \{w \in \cW_K, (w^{(1)}, \ldots, w^{(K)}) = (e_1, \ldots, e_K) \}$, and, for any $i \in [K]$, for any $w$ in the set $\cW^{(e)}_K$, $\zeta^{(i)}(w) = \mu^{(i)}(e) \times \nu_{i,e}(w)$. In other words, $\bze, \tbze$ are obtained from $\bmu,\tbmu$ by specifying the same ordering of the children of each vertex of the tree. Using this characterization, it becomes clear that the $\Gamma$-equivalence is an equivalence relation.
\end{proof}

The following is then immediate.

\begin{proposition}
\label{prop:distrib=proj}
Let $\Gamma \in \cM^{(K)}(\R)$, and let $\bze$ be a family of distributions on $\cW_K$. Let $\bmu$ be its projection. Then there exists a critical family $\tbze$ that is $\Gamma$-equivalent to $\bze$ if and only if there exists a critical family $\tbmu$ that is $\Gamma$-equivalent to $\bmu$.
\end{proposition}

%In the rest of the paper, we will use the word \textit{distribution} for probability distributions on $\cW_K$, and the word \textit{projection} for probability distributions on $\N^K$.

\subsection{Good exponential tiltings}

We exhibit here a sufficient criterion for two projections $\bmu, \tbmu$ to be $\Gamma$-equivalent, similar to \cite[Section $4$]{Jan12} in the monotype case. By Proposition \ref{prop:distrib=proj}, finding a critical projection $\tbmu$ that is $\Gamma$-equivalent to a given projection $\bmu$ is the same as finding a critical distribution $\tbze$ that is $\Gamma$-equivalent to a given distribution $\bze$ whose projection is $\bmu$. We emphasize that the criterion that we will use is only a sufficient condition for two projections to be $\Gamma$-equivalent, and does not fully characterize the $\Gamma$-equivalence, contrary to the monotype case. Therefore, our main result only provides a partial answer to the question of whether there exists a critical distribution that is $\Gamma$-equivalent to a given one.

The main concept of this section is the notion of \textit{exponential tiltings} for projections.

\begin{definition}
Let $\bmu := (\mu^{(1)}, \ldots, \mu^{(K)})$, $\tbmu := (\tmu^{(1)}, \ldots, \tmu^{(K)})$ be two families of projections on $\N^K$. We say that $\tbmu$ is an exponential tilting of $\bmu$ if there exist $2K$ constants $a_1, \ldots, a_K, b_1, \ldots, b_K > 0$ such that, for any $\bk:=(k_1, \ldots, k_K) \in \N^K$, any $i \in [K]$:
\begin{align*}
\tmui(\bk) := a_i \prod_{j=1}^K b_j^{k_j} \mui(\bk).
\end{align*}

Equivalently, for all $i \in [K]$, all $s_1, \ldots, s_K \in [0,1]^K$:
\begin{align*}
\tphii(s_1, \ldots, s_K) = a_i \phii(b_1s_1, \ldots, b_Ks_K),
\end{align*}
where we denote by $\tphii$ the generating function of $\tmui$ for $i \in [K]$.
\end{definition}

It is clear that, if $\tbmu$ is an exponential tilting of $\bmu$, then $\bmu$ is an exponential tilting of $\tbmu$, and that $\bmu$ is entire (resp. irreducible) if and only if $\tbmu$ is entire (resp. irreducible). Furthermore, the fact that $\tmui$ is a probability distribution for all $i \in [K]$ implies that $a_1, \ldots, a_K, b_1, \ldots, b_K$ shall satisfy

\begin{equation}
\label{eq:cases}
\begin{cases} 
 \tilde{\phi}^{(1)}(1, \ldots, 1) = 1 \\ 
 \tilde{\phi}^{(2)}(1, \ldots, 1) = 1 \\
 \qquad \qquad \vdots \\
 \tilde{\phi}^{(k)}(1, \ldots, 1) = 1,  
 \end{cases}
 \end{equation}

\noindent which is equivalent to $a_i^{-1} = \phi^{(i)}(b_1, \ldots, b_K)$ for all $i \in [K]$. In other words, specifying the $(b_i, i \in [K])$ forces the values of the $(a_i, i \in [K])$.

Our first result characterizes a family of exponential tiltings that preserve the distribution of the conditioned multitype trees. 

\begin{definition}
Let $\Gamma \in \cM^{(K)}(\R)$. We say that $(a_i, b_i)_{i \in [K]}$ satisfying \eqref{eq:cases} is a \textit{good exponential tilting} if $\bmu$ and $\tbmu$ are $\Gamma$-equivalent, where $\tbmu$ is the exponential tilting of $\bmu$ obtained from $(a_i, b_i)_{i \in [K]}$.
\end{definition}

The interest of this definition lies in the following result.

\begin{proposition}
\label{prop:goodexptilt}
Let $\Gamma \in \cM^{(K)}(\R)$ and $\{ (a_i, b_i), i \in [K] \} \in \left((\R_+^*)^2\right)^K$ satisfying \eqref{eq:cases}.
Define, for all $i \in [K]$, $c_i := \log(a_i b_i)$. Then, if $$\bc := (c_i, i \in [K]) \in (Ker \, \Gamma)^{\perp},$$ 
we have that $\{ (a_i, b_i), i \in [K] \}$ is a good exponential tilting.
\end{proposition}

\begin{remark}
Observe that it is not an equivalence, as there may be good exponential tiltings that do not satisfy $\bc \in (Ker \, \Gamma)^{\perp}$.
\end{remark}

Proposition \ref{prop:goodexptilt} makes clear the dependency in $\Gamma$ of the notion of good exponential tilting: different matrices $\Gamma$ clearly provide different notions of good exponential tiltings.

As a corollary of Proposition \ref{prop:goodexptilt}, we have $rk(\Gamma)$ degrees of freedom in the choice of a good tilting. In particular this is minimum when $rk(\Gamma)=1$, in which case we can restrict ourselves to conditionings of the form
\begin{align*}
\sum_{i=1}^K \gamma_i N_i(\cT) = Q
\end{align*}
for some constant $Q \in \R$, that is (by \ref{condition}), $\Gamma \in \cM_{1,K}(\N^*)$. As an example, Pénisson \cite{Pen14} and Abraham-Delmas-Guo \cite{ADG18} consider the case $\Gamma=Id_K$. The existence of a critical projection $\Gamma$-equivalent to $\bmu$ can (under our assumptions) be deduced from the same result for $\Gamma=(1 0 \ldots 0) \in \cM_{1,K}$.

We now prove Proposition \ref{prop:goodexptilt}.

\begin{proof}[Proof of Proposition \ref{prop:goodexptilt}]
Fix $L \geq 1$ and $\Gamma \in \cM_{L,K}(\R)$. Let $j \in [K]$ and $\bg = (g_1, \ldots, g_L)$ such that $\P (\cT^{(j)} \text{ satisfies } \eqref{eq:gamma})>0$. Let $\bT^{(j)}_{\Gamma,\bg}$ be the set of trees $T$ with root label $j$ satisfying
\begin{equation}
\label{eq:GammaT}
\Gamma \begin{pmatrix}
N_1(T)\\
\vdots \\
N_K(T)
\end{pmatrix}
= 
\begin{pmatrix}
g_1\\
\vdots \\
g_L
\end{pmatrix}.
\end{equation}

For a tree $T$, a vertex $v \in T$ and $i \in [K]$, let $k_v^{(i)}(T)$ be the number of children of $v$ in $T$ with label $i$. For all $T \in \bT^{(j)}_{\Gamma,\bg}$, we have that
\begin{align*}
\P\left( \cT^{(j)}_{\Gamma,\bg} = T \right) = \frac{w(T)}{Z_{\Gamma,\bg}}, 
\end{align*}
where 
\begin{align*}
w(T) = \prod_{i \in [K]}\prod_{v \in T, \ell(v)=i} \mu^{(i)}\left(k_v^{(1)}(T), \ldots, k_v^{(K)}(T)\right)
\end{align*}
and
\begin{align*}
Z_{\Gamma,\bg} = \sum_{U \in \bT^{(j)}_{\Gamma, \bg}} w(U).
\end{align*}

On the other hand, we have

\begin{align*}
\P\left( \tilde{\cT}^{(j)}_{\Gamma,\bg} = T \right) = \frac{\tilde{w}(T)}{\tilde{Z}_{\Gamma,\bg}}, 
\end{align*}
where 
\begin{align*}
\tilde{w}(T) &= \prod_{i \in [K]}\prod_{v \in T, \ell(v)=i} a_i \mu^{(i)}\left(k_v^{(1)}(T), \ldots, k_v^{(K)}(T)\right) \prod_{r=1}^K b_r^{k_v^{(r)}(T)}\\
&=\prod_{i \in [K]} a_i^{N_i(T)} b_i^{N_i(T)} b_j^{-1} w(T)
\end{align*}
(the factor $b_j^{-1}$ corresponds to the root label), and
\begin{align*}
\tilde{Z}_{\Gamma,\bg} = \sum_{U \in \bT^{(j)}_{\Gamma, \bg}} \tilde{w}(U).
\end{align*}

Hence, $\cT^{(j)}_{\Gamma,\bg} \overset{(d)}{=} \tilde{\cT}^{(j)}_{\Gamma,\bg}$ if and only if $\prod_{i \in [K]} a_i^{N_i(T)} b_i^{N_i(T)}$ is constant on $\bT_{\Gamma,\bg}^{(j)}$. This is equivalent to
\begin{align*}
\langle \bc, \bN(T) \rangle \text{ is constant,}
\end{align*}
where $\langle \cdot, \cdot \rangle$ is the usual scalar product on $\R^K$, $\bc = (c_1, \ldots, c_K)$ and $\bN(T)=(N_1(T), \ldots, N_K(T))$.

In particular, $\{ (a_i, b_i), i \in [K] \}$ is a good exponential tilting if $\bc \in S_{\Gamma}$, where

\begin{align*}
S_\Gamma := \left\{ \bc \in \R^K, \forall \bx, \by \in \N^K, \Gamma \bx = \Gamma \by \Rightarrow \langle \bc, \bx \rangle = \langle \bc, \by \rangle  \right\}.
\end{align*}

Since $\N-\N=\Z$, it is clear that
\begin{align*}
S_\Gamma &= \left\{ \bc \in \R^K, \forall \bx \in \Q^K, \Gamma \bx = 0 \Rightarrow \langle \bc, \bx \rangle = 0 \right\}\\
&= \left\{ \bc \in \R^K, Ker \, \Gamma \cap \Q^K \subseteq Ker F_\bc \right\},
\end{align*}
where $F_\bc \in (\R^K)^*: \bx \mapsto \langle \bc, \bx \rangle$ is the linear form associated to $\bc$.

Using the fact that $dim_\Q (Ker \, \Gamma) = dim_\R (Ker \, \Gamma)$, we get that

\begin{align*}
S_\Gamma &= \left\{ \bc \in \R^K, Ker \, \Gamma \subseteq Ker F_\bc \right\}\\
&= (Ker \, \Gamma)^{\perp}.
\end{align*}
\end{proof}

In particular, $S_\Gamma$ is a vector space of dimension $dim \, S_\Gamma = rk(\Gamma)$.

\section{Existence of a critical exponential tilting}
\label{sec:proofs}

We prove here the first part of our main theorem, Theorem \ref{thm:main}, stating the existence of a critical exponential tilting of any offspring entire distribution, under assumptions \ref{cond:entire}-\ref{cond:escape} and \ref{condition}. To this end, by Proposition \ref{prop:goodexptilt}, we can restrict ourselves to the case where $(Ker \, \Gamma)^\perp = \R \ga$ for some $\ga$ satisfying \ref{condition}. Without loss of generality, we can assume that $\gamma_1=1=\min\{ \gamma_i, i \in [K] \}$. In other words, without loss of generality, $\Gamma$ is of the form
\begin{align*}
\Gamma = (1 \, \gamma_2 \, \ldots \gamma_K) \in (\N^*)^K.
\end{align*}

\subsection{The setting}
\label{ssec:setting}

We fix the type of the root of our trees (say, $j \in [K]$) and condition our trees on their total weighted number of vertices $N_1(\cT^{(j)}) + \sum_{i=2}^K \ga_i N_i(\cT^{(j)})$. Hence, we have, with the notation of Section \ref{sec:tiltings}:
\begin{align*}
S_\Gamma = \R \begin{pmatrix}
1\\
\gamma_2\\
\vdots \\
\gamma_K
\end{pmatrix}.
\end{align*}

Take $\bc \in S_\Gamma$, and set $\beta=\exp(c_1)$. By assumption, $c_i/\gamma_i=\log(\beta)$ for all $i \in [K]$. The system \eqref{eq:cases} becomes

\begin{equation}
\label{eq:casesbis}
\begin{cases} 
 \beta \frac{\phi^{(1)}(b_1, \ldots, b_K)}{b_1} = 1 \\ 
 \beta \left( \frac{\phi^{(2)}(b_1, \ldots, b_K)}{b_2} \right)^{1/\gamma_2} = 1 \\
 \qquad \qquad \vdots \\
 \beta \left(\frac{\phi^{(K)}(b_1, \ldots, b_K)}{b_K}\right)^{1/\gamma_K} = 1. \\
 \end{cases}
\end{equation}

%To prove our main theorem, we need to show that, under the assumptions \ref{cond:entire}-\ref{cond:escape}, there always exists a good critical exponential tilting. Our proof shows its existence but does not construct it explicitly. In particular, $(b_1, \ldots, b_K)$ could potentially be arbitrarily far from $(1, \ldots, 1)$.

\subsection{The tilted mean matrix.}

We start by connecting the spectral radius of the tilted mean matrix to the original one. For any $\bb := (b_1, \ldots, b_K) \in (0,+\infty)^K$, we denote by $\tilde{\rho}(\bb)$ the spectral radius of the mean matrix $\tilde{M}$ of the tilted projection associated to $\{ (a_i, b_i), i \in [K] \}$ (recall that, by definition, the $a_i$'s are uniquely defined by the $b_i$'s).

\begin{lemma}
\label{lem:specrad}
For any $\bb := (b_1, \ldots, b_k)$ satisfying \eqref{eq:casesbis}, the spectral radius $\tilde{\rho}(\bb)$ of $\tilde{M}$ satisfies
\begin{align*}
\tilde{\rho}(\bb) = \rho(M'),
\end{align*}
where $M'=\left( \beta^{\gamma_i} \frac{\partial \phi^{(i)}}{\partial x_j} \left( \bb \right) \right)_{1 \leq i,j \leq K}$ and $\rho(M)$ stands for the spectral radius of a matrix $M$.
\end{lemma}

\begin{proof}[Proof of Lemma \ref{lem:specrad}]
It is clear that, for all $1 \leq i,j \leq k$:
\begin{align*}
\tilde{M}_{i,j} &= \frac{\partial \tilde{\phi}^{(i)}}{\partial x_j} (1, \ldots, 1)\\ 
&= \frac{\beta^{\gamma_i}}{b_i} b_j \frac{\partial \phi^{(i)}}{\partial x_j} (\bb) = \frac{b_j}{b_i} M'_{i,j}.
\end{align*}
In particular, we have
\begin{align*}
\tilde{M} = P^{-1}M'P,
\end{align*}
where $P=diag(b_1, \ldots, b_K)$ is the diagonal matrix with $P_{i,i}=b_i$ for all $i \in [K]$. Since $\tilde{M}$ and $M'$ are similar, they have the same eigenvalues, and thus the same spectral radius.
\end{proof}

\subsection{Proof of the main result.}

We now turn to the proof of the first part of Theorem \ref{thm:main}. Let us explain the strategy of the proof. We consider the set $A_+ := \{ \bm{0} \} \cup \{ \bb \in (0,+\infty)^K, \eqref{eq:casesbis} \text{ holds for some } \beta > 0 \}$. We prove that, in a neighbourghood of $\bm{0}$ in $(\R+)^K$, there exists a nontrivial continuous simple curve $\cC$ containing $\bm{0}$ such that $\cC \subseteq A_+$. Furthermore, $\tilde{\rho}(\bb)$ goes to $0$ as $\bb \in A_+ \setminus \{ \bm{0} \}$ goes to $\bm{0}$. Then, the idea is roughly speaking to follow this curve $\cC$ starting from $\bm{0}$ and prove that it contains a point $\bb \in \R_+^K$ at which $\tilde{\rho}(\bb)=1$.

\subsubsection{Around the origin}

Our first goal is to study the set $A_+$ in a neighbourhood of $\bm{0}$. To this end, we introduce for all $1 \leq i, j \leq K$:

\begin{align*}
G_{i,j}: (b_1, \ldots, b_K) \mapsto b_j^{\gamma_i} \left(\phi^{(i)}(b_1, \ldots, b_K)\right)^{\gamma_j} - b_i^{\gamma_j} \left(\phi^{(j)} (b_1, \ldots, b_K)\right)^{\gamma_i}.
\end{align*}

\noindent In particular, $G_{i,j} = -G_{j,i}$ for all $i,j \in [K]$.
Since we have assumed that the $\phi^{(i)}$'s are all entire (Assumption \ref{cond:entire}), for all $i,j \in [K]$, $G_{i,j}$ can be extended on all $\R^K$. Clearly all these functions are holomorphic (since $\gamma_i \in \N^*$ for all $i \in [K]$ by \ref{condition}), and 
\begin{align*}
A_+ = \left( \{ \bm{0} \} \cup (0,\infty)^K \right) \cap \bigcap_{1 \leq i < j \leq K} \left\{ G_{i,j}^{-1}(0) \right\}.
\end{align*}

Our first result is the existence of the curve $\cC$ mentioned above. In other words, close to $\bm{0}$, $A_+$ is the graph of a function. 
It is useful to define the extension of $A_+$ to $\R^K$, and set

\begin{align*}
A := \R^K \cap \bigcap_{i,j \in [K]} G_{i,j}^{-1}(\{0\}).
\end{align*} 

\begin{theorem}
\label{thm:aroundzero}
There exists a function $\psi: \R \rightarrow \R^{K-1}$ defined on an open neighbourhood $V$ of $\bm{0}$ in $\R^K$ such that, for $(b_1, \ldots, b_K) \in V$, 
\begin{align*}
(b_1, \ldots, b_K) \in A \Leftrightarrow  (b_2, \ldots, b_K) = \psi(b_1).
\end{align*}

\noindent Furthermore, we have for all $2 \leq j \leq K$:
$\psi_j^{[s]}(\bm{0})=0$ for $s \in \{ 0, \ldots, \gamma_j-1 \}$ and
\begin{align*}
\psi_j^{[\gamma_j]}(\bm{0}) = \gamma_j! \left(\phi^{(1)}(\bm{0})\right)^{-\gamma_j} \phi^{(j)}(\bm{0}) 
\end{align*}
where $f^{[s]}$ denotes the $s$-th derivative of $f$ and $\psi_j$ is the $(j-1)$-st coordinate of $\psi$. In particular, $\psi_j^{[\gamma_j]}>0$.
\end{theorem}

This ensures that the connected component of $A$ containing $\bm{0}$ is a simple curve around $\bm{0}$ and that, in a neighbourhood $V$ of $\bm{0}$ in $\R^K$, for any $(b_1, \ldots, b_K) \in A \cap V$, we have $(b_1, \ldots, b_K)=\bm{0}$ or all $b_i$'s have the same sign.

\begin{proof}[Proof of Theorem \ref{thm:aroundzero}]
We apply the implicit function theorem to the function
\begin{align*}
G: (b_1, \ldots, b_K) \in \R^K \mapsto \left( G_{1,2}(b_1,\ldots, b_K), \ldots, G_{1, K}(b_1, \ldots, b_K) \right) \in \R^{K-1}.
\end{align*}

This function is clearly $C^{\infty}$ on $\R^K$. Observe that $G(\bm{0})=\bm{0}$. Furthermore, we have for all $2 \leq j,j' \leq K$:

$$
\frac{\partial G_{1,j}}{\partial b_{j'}}(\bm{0}) = \left\{
    \begin{array}{ll}
        \left(\phi^{(1)}(\bm{0})\right)^{\gamma_j} & \mbox{if } j=j' \\
        0 & \mbox{otherwise.}
    \end{array}
\right.
$$

In particular, by \ref{cond:jamaiszero}, $\phi^{(1)}(\bm{0})>0$ and the Jacobian matrix at $\bm{0}$ is diagonal and invertible. The first part of the result follows by the implicit function theorem. 
The second part follows directly from the chain rule and the computation of $\frac{\partial^s G_{1,j}}{\partial b_1^s}$ for $1 \leq s \leq \gamma_j$.
\end{proof}

Observe that this proof is based on Assumption \ref{condition} and the fact that $\gamma_1=1$. From now on, we denote by $\cC \subseteq \R_+^K$ the connected component of $A_+$ containing $\bm{0}$. By Theorem \ref{thm:aroundzero}, $\cC \neq \{ \bm{0} \}$. It is interesting to notice that, in general, the set $A_+$ is not connected. We now prove that, for $(b_1, \ldots, b_K) \in \cC$ close enough to $\bm{0}$, the associated tilted spectral radius is at most $1$.

\begin{lemma}
\label{lem:rho<1}
There exists $r>0$ such that, for all $\bb := (b_1, \ldots, b_K) \in A_+ \setminus \{ \bm{0} \}$ such that $\sup_{i \in [K]} b_i \leq r$, $\tilde{\rho}(\bb) < 1$.
\end{lemma}

\begin{proof}[Proof of Lemma \ref{lem:rho<1}]
Define $\rho := \tilde{\rho}(1, \ldots, 1)$, the spectral radius of the original mean matrix. Observe that, by Assumption \ref{cond:jamaiszero} and \eqref{eq:casesbis}, for $\bb \in A_+ \setminus \{ \bm{0} \}$ close enough to $\bm{0}$, we have $\sup_{i \in [K]}\beta^{\gamma_i}<\frac{1}{2}\rho^{-1}$. Furthermore, if $b_i \leq 1$ for all $i \in [K]$, then 
\begin{align*}
\rho \left( \left( \frac{\partial \phi^{(i)}}{\partial b_j}(b_1, \ldots, b_K) \right)_{i,j \in [K]} \right) \leq \rho,
\end{align*}
since the spectral radius is a nondecreasing function of each coordinate (provided that they are all nonnegative). The result follows by Lemma \ref{lem:specrad}.
\end{proof}

The interest of this lemma is the following: since $\tilde{\rho}$ is continuous on $\cC$, we only need to show that there exists $\bb \in \cC$ such that $\tilde{\rho}(\bb) \geq 1$ (or directly $\bb \in A_+$ such that $\tilde{\rho}(\bb)=1$). To this end, we consider different cases, depending on whether of not $\overline{\cC}$ is compact in $\R_+^K$.

A first result of importance is the fact that $\cC$ cannot escape the cone $(0,+\infty)^K$, in the following sense. Recall that $A := \bigcap_{i,j \in [K]} G_{i,j}^{-1}(\{0\})$.

\begin{lemma}
\label{lem:noborder}
Let $\bb \in [0,\infty)^K \cap A$ such that there exists $i \in [K]$ for which $b_i=0$. Then, $b_i=0$ for all $i \in [K]$.
\end{lemma}

\begin{proof}
Assume without loss of generality that $b_1=0$. For all $2 \leq j \leq K$, since $G_{1,j}(\bb)=0$ and $\phi^{(1)}(b_1, \ldots, b_K) \neq 0$ (by Assumption \ref{cond:jamaiszero}), we have $b_j=0$.
\end{proof}

\subsubsection{If $\overline{\cC}$ is not compact}

Let us first consider the case where $\overline{\cC}$ is not compact.

\begin{theorem}
\label{thm:escapesfromcompact}
Assume that $\overline{\cC}$ is not compact. Then $\cC$ contains a good critical exponential tilting.
\end{theorem}

\begin{proof}
It is clear that $\tilde{\rho}$ is continuous on $\cC$. Furthermore, by Lemma \ref{lem:rho<1}, for $(b_1, \ldots, b_K) \in A_+ \setminus \{ \bm{0} \}$ close enough to $\bm{0}$, we have $\tilde{\rho}(b_1, \ldots, b_K) < 1$. Hence, it suffices to prove that, for $(b_1, \ldots, b_K) \in \cC$ far enough from $\bm{0}$, we have $\tilde{\rho}(b_1, \ldots, b_K) \geq 1$. By Lemma \ref{lem:noborder}, we only need to consider points in the cone $\R_+^K$. To this end, assume without loss of generality that $b_1 \rightarrow +\infty$ on $\cC$ (possibly along a subsequence). By Assumption \ref{cond:escape}, we have, for $b_1$ large enough, uniformly in $b_2,\ldots, b_K \in \R_+$,
\begin{align*}
\frac{\partial \phi^{(1)}(b_1, \ldots, b_K)}{\partial b_1} \geq \frac{\phi^{(1)}(b_1, b_2, \ldots, b_K)}{b_1}.
\end{align*}
\noindent Now observe that the spectral radius of a matrix is nondecreasing in all coefficients, and the spectral radius of the matrix $$\left(\frac{\partial \phi^{(1)}(b_1, b_2 \ldots, b_K)}{\partial b_1} \mathbbm{1}_{i=j=1}\right)_{1 \leq i,j \leq K}$$ is $\frac{\partial \phi^{(1)}(b_1, b_2 \ldots, b_K)}{\partial b_1}$.
We get therefore that $$\tilde{\rho}(b_1, \ldots, b_K) \geq \beta^{\gamma_1} \frac{\partial \phi^{(1)}(b_1, b_2 \ldots, b_K)}{\partial b_1} \geq \beta^{\gamma_1} \frac{\phi^{(1)}(b_1, b_2, \ldots, b_K)}{b_1} = 1.$$ The result follows.
\end{proof}

\subsubsection{The set of degenerate points}

Assume now that $\overline{\cC}$ is compact. The rest of the proof is based on the study of the set of degenerate points, that is, points around which $\cC$ is not the graph of a function of one of the $b_i$'s.

Let us first introduce some functions, slightly different from the $G_{i,j}$'s. For all $i, j \in [K]$, define
\begin{align*}
H_{i,j}(b_1, \ldots, b_K) = b_j^{1/\gamma_j} \left( \phi^{(i)}(b_1, \ldots, b_K) \right)^{1/\gamma_i} - b_i^{1/\gamma_i} \left( \phi^{(j)}(b_1, \ldots, b_K) \right)^{1/\gamma_j},
\end{align*}
and the associated Jacobian matrices $(I^{(i)}_{i \in [K]}) \in \cM_{K-1, K-1}$ defined as
\begin{align*}
I^{(i)}(b_1, \ldots, b_K) = \left( \frac{\partial H_{i,j}}{\partial b_{j'}}(b_1, \ldots, b_K) \right)_{j,j' \neq i}.
\end{align*}

Observe in particular that $A_+ = \{ \bm{0}\} \cup \bigcap_{i,j \in [K]} H_{i,j}^{-1}(\{0\})$. For convenience, we still label the rows and columns of $I^{(i)}$ by $[K] \backslash \{ i \}$ and not $[K-1]$. We also define the set of degenerate points as follows:
\begin{align*}
E := \left\{ \bb \in (0,\infty)^K, \, \forall i \in [K], \, det \, I^{(i)}(\bb) = 0 \right\}
\end{align*}

Our proof is divided in several parts, which we informally describe. First, we show that the value of $\tilde{\rho}$ at any degenerate point in $A_+$ is $\geq 1$. Second, we show that, if $\overline{\cC}$ is compact, then $\overline{\cC}$ necessarily contains a degenerate point $x$. Studying separately the cases $x \in \cC$ and $x \notin \cC$, we complete the proof of the existence of a good critical exponential tilting.

\begin{theorem}
\label{thm:rhoatleast1}
Let $\bb \in E$. Then, $\tilde{\rho}(\bb) \geq 1$.
\end{theorem}

As a corollary, we obtain the following:

\begin{cor}
\label{cor:ifchasadegeneratepoint}
Assume that $\cC \cap E \neq \emptyset$. Then, there exists a good critical exponential tilting. 
\end{cor}

\begin{proof}[Proof of Corollary \ref{cor:ifchasadegeneratepoint}]
This is a simple consequence of Theorem \ref{thm:rhoatleast1} along with the continuity of $\tilde{\rho}$ on $\cC$ along with Lemmas \ref{lem:rho<1} and \ref{lem:noborder}.
\end{proof}

The idea of the proof of Theorem \ref{thm:rhoatleast1} is to exhibit an eigenvector of $\tilde{M}$ whose associated eigenvalue is $1$.

\begin{proof}[Proof of Theorem \ref{thm:rhoatleast1}]
Let $\bb \in E$, and recall that the matrix $\tilde{M}$ is defined as 
\begin{align*}
\tilde{M}_{i,j} = \frac{\partial \tphii}{\partial b_j}(1,\ldots,1).
\end{align*}

Recall that, since $\bb \in \cC$, there exists $\beta > 0$ such that, for all $i \in [K]$,
\begin{align*}
\beta \left(\frac{\phi^{(i)}(\bb)}{b_i}\right)^{1/\gamma_i} = 1 \text{ (see Section \ref{ssec:setting}).}
\end{align*}
Our claim is the following: if $\bb \in E$, then there exists a vector $Z$ satisfying 
\begin{equation}
\label{eq:toprove}
\tilde{M} Z = Z.
\end{equation}

\noindent In particular, $1$ is in the spectrum of $\tilde{M}$ and necessarily $\tilde{\rho} \geq 1$.

We first compute the matrix $I^{(i)}(\bb)$ at a point of $E$. Set for convenience $\delta_i = 1/\gamma_i$ for $i \in [K]$. In what follows, since it is clear by the context, all functions are taken at the point $\bb$.
By definition, for any $j,j' \neq i$, we have
\begin{align*}
I^{(i)}(\bb)_{j,j'} &= \frac{\partial H_{i,j}}{\partial b_{j'}}\\
&= b_j^{\delta_j} \frac{\partial \left[\phi^{(i)}\right]^{\delta_i}}{\partial b_{j'}} - b_i^{\delta_i} \frac{\partial \left[\phi^{(j)}\right]^{\delta_j}}{\partial b_{j'}} + \mathbbm{1}_{j=j'} \delta_j b_j^{\delta_j-1} \left(\phi^{(i)}\right)^{\delta_i}.
\end{align*}

We now choose, for each $i \in [K]$, an eigenvector $z^{(i)} := (z^{(i)}_j, j \neq i) \in Ker \, I^{(i)} \backslash \{ 0 \}$. This vector exists by assumption, since $(b_1, \ldots, b_K) \in E$. Again, we label its coordinates by $[K] \backslash \{ i \}$ for convenience.  We will construct a $1$-eigenvector $Z$ of $\tilde{M}$ as a linear combination of the $z^{(i)}$'s. To this end, let $(d_1, \ldots, d_K) \in \R^K \backslash \{ \bm{0} \}$ such that
\begin{align}
\label{eq:di}
\sum_{j=1}^K \frac{d_j}{b_j^{\delta_j}} \sum_{\substack{i=1 \\ i \neq j}}^K \frac{\partial \left[\phi^{(j)}\right]^{\delta_j}}{\partial b_i} z_i^{(j)} = 0.
\end{align}
Define the vector $Y$ whose coordinates satisfy
\begin{align*}
Y_i = \sum_{\substack{j=1 \\ j \neq i}}^K d_j z_i^{(j)}.
\end{align*}

\begin{lemma}
\label{lem:lastlemma}
The vector $Z := P^{-1}Y$ is a $1$-eigenvector of the matrix $\tilde{M}$, where $P:=diag(b_1,\ldots,b_K)$.
\end{lemma}

In particular, this immediately implies Theorem \ref{thm:rhoatleast1}. It is quite clear that we can choose $(d_1, \ldots, d_K)$ so that $Y$ is not the $0$ vector. Indeed, the following holds:
\begin{itemize}
\item if, for all $i \neq j$, $z_i^{(j)}=0$, then any $(d_1, \ldots, d_K) \in \R^K$ satisfies \eqref{eq:di}. In particular $(1,0, \ldots, 0)$ works, and there exists $i \in [K]$ such that $Y_i := z_i^{(1)} \neq 0$ (because $z^{(1)}$ is an eigenvector);
\item otherwise, let $i \neq j$ such that $z_i^{(j)} \neq 0$, and assume without loss of generality that $j=1$. 
\begin{itemize}
\item If $d_1=0$ for all $(d_1, \ldots, d_K)$ satisfying \eqref{eq:di}, then it means that the set of $(d_1, \ldots, d_K)$ satisfying \eqref{eq:di} is $\{ 0 \} \times \R^{K-1}$. Then, let $\ell \in [K]$ such that $z_\ell^{(2)} \neq 0$ and choose $(d_1, \ldots, d_K) = (0,1,0,\ldots,0)$. It satisfies \eqref{eq:di} and $Y_\ell = z_\ell^{(2)} \neq 0$.
\item otherwise, let $(d_1, \ldots, d_K)$ satisfying \eqref{eq:di} with $d_1 \neq 0$ and $d_\ell = 0$ for $\ell \notin \{ 1, i \}$ (such a solution exists since the space of solutions has dimension $\geq K-1$). We have in particular $Y_i = d_1 z_i^{(1)} \neq 0$.
\end{itemize}
\end{itemize}

\begin{proof}[Proof of Lemma \ref{lem:lastlemma}]
For all $i \neq j \in [K]$, by definition of $z^{(j)}$, we have
\begin{align}
\label{eq:zij}
0 &= \sum_{i'=1, i' \neq j}^K I^{(j)}_{i,i'} z_{i'}^{(j)} \nonumber\\
&= \sum_{i'=1, i' \neq j}^K \left( b_i^{\delta_i} \frac{\partial \left[\phi^{(j)}\right]^{\delta_j}}{\partial b_{i'}} - b_j^{\delta_j} \frac{\partial \left[\phi^{(i)}\right]^{\delta_i}}{\partial b_{i'}} \right) z_{i'}^{(j)} + \delta_i b_i^{\delta_i-1} \left(\phi^{(j)}\right)^{\delta_j} z_i^{(j)} \nonumber\\
&= \sum_{i'=1, i' \neq j}^K \left(  b_i^{\delta_i} \frac{\partial \left[\phi^{(j)}\right]^{\delta_j}}{\partial b_{i'}} - b_j^{\delta_j} \frac{\partial \left[\phi^{(i)}\right]^{\delta_i}}{\partial b_{i'}} \right) z_{i'}^{(j)} + \beta^{-1} \delta_i b_i^{\delta_i-1} b_j^{\delta_j} z_i^{(j)}.
\end{align}

For any $i \in [K]$, we have:

\begin{align*}
\sum_{i'=1}^K \frac{\partial \left[\phi^{(i)}\right]^{\delta_i}}{\partial b_{i'}} Y_{i'} &= \sum_{i'=1}^K \frac{\partial \left[\phi^{(i)}\right]^{\delta_i}}{\partial b_{i'}} \sum_{j=1, j \neq i'}^K d_j z_{i'}^{(j)}\\
&= \sum_{j=1}^K d_j \sum_{i'=1, i' \neq j}^K \frac{\partial \left[\phi^{(i)}\right]^{\delta_i}}{\partial b_{i'}} z_{i'}^{(j)}\\
&= \sum_{j=1}^K \frac{d_j}{b_j^{\delta_j}}  \left( \sum_{i'=1, i' \neq j}^K b_i^{\delta_i} \frac{\partial \left[\phi^{(j)}\right]^{\delta_j}}{\partial b_{i'}} z_{i'}^{(j)} \right) + \sum_{j=1}^K \frac{d_j}{b_j^{\delta_j}} \beta^{-1}\delta_i b_i^{\delta_i-1} b_j^{\delta_j} z_i^{(j)},
\end{align*}
by \eqref{eq:zij}. Now observe that
\begin{align*}
\sum_{j=1}^K \frac{d_j}{b_j^{\delta_j}}  \left( \sum_{i'=1, i' \neq j}^K b_i^{\delta_i} \frac{\partial \left[\phi^{(j)}\right]^{\delta_j}}{\partial b_{i'}} z_{i'}^{(j)} \right) &= b_i^{\delta_i} \sum_{j=1}^K \frac{d_j}{b_j^{\delta_j}} \sum_{i'=1, i' \neq j}^K \frac{\partial \left[\phi^{(j)}\right]^{\delta_j}}{\partial b_{i'}} z_{i'}^{(j)} = 0,
\end{align*}
by definition of $(d_1, \ldots, d_K)$. We are thus left with
\begin{align*}
\sum_{i'=1}^K \frac{\partial \left[\phi^{(i)}\right]^{\delta_i}}{\partial b_{i'}} Y_{i'} &= \beta^{-1} \sum_{j=1}^K d_j \delta_i b_i^{\delta_i-1} z_i^{(j)}\\
&= \beta^{-1} \delta_i b_i^{\delta_i-1} Y_i,
\end{align*}
which can be rewritten
\begin{align*}
\sum_{i'=1}^K \frac{\partial \phi^{(i)}}{\partial b_{i'}} Y_{i'} = \beta^{-\frac{1}{\delta_i}} Y_i.
\end{align*}

This implies that $M' Y = Y$, where $M'$ is the matrix defined in Lemma \ref{lem:specrad}. By Lemma \ref{lem:specrad} again, it is equivalent to saying that $\tilde{M} P^{-1} Y = P^{-1} Y$.
\end{proof}
\end{proof}

\subsubsection{If $\cC \cap E=\emptyset$.}

The last case to consider is the case where $\cC$ does not contain any element of $E$. By the implicit function theorem, $\cC$ is locally, around each of the points of $\cC \backslash \{ \bm{0} \}$, the graph of a function of $b_i$ for some $i \in [K]$.

\begin{proposition}
\label{prop:ifnodegeneratepoint}
Assume that $\overline{\cC}$ is compact and that $\cC \cap E = \emptyset$. Then, there exists a good exponential tilting in $A_+$.
\end{proposition}

\begin{proof}
Since $\cC \cap E=\emptyset$, for any $\bb \in \overline{\cC} \cap (\R_+^*)^K$, there exists $i \in [K]$ such that $det I^{(i)}(\bb) \neq 0$. Then, $\cC$ is a $1$-dimensional connected manifold with boundary $\partial \cC \subset \R_+^K \backslash (\R_+^*)^K$. It is known that then it is homeomorphic to either $\R,\R_+$, the circle $\mathbb{S}^1$ or the interval $[0,1]$. Since its boundary contains $\bm{0}$, $\partial \cC$ is nonempty and $\cC$ is homeomorphic to either $[0,1]$ or $\R+$. If there exists a homeomorphism $f:\cC \rightarrow [0,1]$, then one can assume without loss of generality that $f(\bm{0})=0$. In this case, let $x:=f^{-1}(1)$. Necessarily, by Lemma \ref{lem:noborder}, $x \in (\R_+^*)^K$ and $det I^{(i)}(x) = 0$ for all $i \in [K]$. Hence, $x \in \cC \cap E$, which contradicts our assumption. Therefore, there exists a homeomorphism $f:\cC \rightarrow \R_+$. Clearly, $f(\bm{0})=0$. Consider the sequence $(x_n)_{n \geq 1} := (f^{-1}(n))_{n \geq 1}$. Since $\overline{\cC}$ is compact, $(x_n)_{n \geq 1}$ has an accumulation point in $\R_+^K$, say $x_\infty$. Furthermore, $x_\infty \in \overline{\cC} \cap (\R_+^*)^K$ by Lemma \ref{lem:noborder}. Indeed, by Theorem \ref{thm:aroundzero}, $x_\infty \neq \bm{0}$. In addition, since $\cC$ is a manifold with boundary $\partial \cC = \{ \bm{0} \}$, necessarily $x_\infty \notin \cC$. In particular, $det I^{(i)}(x_\infty) = 0$ for all $i \in [K]$ and $x_\infty \in E$.

Observe now that, since $A$ is closed, we have that $x_\infty \in A$. Thus, if $\tilde{\rho}(x_\infty)=1$, then $x_\infty$ corresponds to a good exponential tilting. If $\tilde{\rho}(x_\infty)\neq 1$, then by Theorem \ref{thm:rhoatleast1} we have $\tilde{\rho}(x_\infty) > 1$. By definition of $x_\infty$, there exists $n>0$ such that  $\tilde{\rho}(x_n) > 1$. We conclude by continuity of $\tilde{\rho}$ and Lemmas \ref{lem:rho<1} and \ref{lem:noborder}.
\end{proof}

We can finally prove our main theorem.

\begin{proof}[Proof of Theorem \ref{thm:main}]
It is a consequence of Theorem \ref{thm:escapesfromcompact}, Corollary \ref{cor:ifchasadegeneratepoint} and Proposition \ref{prop:ifnodegeneratepoint}.
\end{proof}

\section{Convergence of conditioned BGW trees}
\label{sec:convergenceoftrees}

In this final section, we prove Corollary \ref{cor:loclim} as a consequence of Theorem \ref{thm:main}, and deduce from it the second part of Theorem \ref{thm:main}. Let $\bze$ be an offspring distribution satisfying \ref{cond:entire}-\ref{cond:escape}, and let $\Gamma := (\gamma_1, \ldots, \gamma_K)$ satisfying \ref{condition}. The main idea is that, by Theorem \ref{thm:main}, there exists a critical distribution equivalent to $\bze$. We then invoke \cite[Theorem $3.1$]{Ste18} to conclude the proof.

\subsection{Kesten-like trees}

We construct here the infinite discrete trees that appear as local limits of critical multitype BGW trees. It turns out that they all share a common structure: a unique end (infinite spine), on which are grafted independent multitype trees that are identically distributed conditionally on their root label. In regard of Kesten's seminal work \cite{Kes86}, we will call these trees Kesten-like trees. This multitype construction was first introduced in \cite{KLPR97}, see also \cite[Proposition $3.1$]{Ste18} for a proof in the broader case of mutitype forests.

Let $\bze$ be an irreducible $K$-type critical distribution. The Perron-Frobenius theorem ensures that, under this irreducibility assumption, $M$ has a real eigenvalue $\rho>0$ of maximal modulus which is simple, and every $\rho$-eigenvector of $M$ has only non-zero coordinates, all of the same sign. Denote by $\br:=(r_1, \ldots, r_K)$ the renormalized right $1$-eigenvector of the mean matrix $M$. Denote by $\bhz:=(\hz^{1)}, \ldots, \hz^{(K)})$ the biased family of distributions defined as:
\begin{align*}
\forall j \in [K], \forall \bx \in \cW_K, \hz^{(j)}(\bx) = \frac{1}{r_j} \sum_{\ell=1}^{|\bx|} r_{x_\ell} \zeta^{(j)}(\bx),
\end{align*}
where $|\bx|$ denotes the length of $\bx$. In particular, $\hz^{(j)}(\emptyset) = 0$.

\begin{definition}
Let $\bze$ be a $K$-type critical distribution
Given a type $i \in [K]$, we define the tree $\cT^{(i)}_*$ as follows: it is made of a spine, which is an infinite branch starting from the root which has label $i$. On this infinite branch, vertices have distribution $\bhz$. Given an element $v$ of the spine, denote by $\bw_v$ its ordered list of offspring types. Then, the probability that the child of $v$ belonging to the infinite spine is $vj$ (that is, the $j$-th of its children) is proportional to $r_{\ell(vj)}$ - that is, equal to
\begin{align*}
\frac{r_{\ell(vj)}}{\sum_{i=1}^{|\bw_v|} r_{\ell(vi)}}.
\end{align*}

Finally, on any offspring of type $j$ of a vertex of the spine that is not itself on the spine, we graft a tree $\cT^{(j)}$ that is independent of the rest of the tree.
\end{definition}

In the monotype case, the child of a vertex on the spine that will be itself on the spine is just chosen uniformly at random. Observe also that, since $\hz^{(j)}(\emptyset)=0$ for all $j \in [K]$, the spine is indeed infinite.

We mention the following local limit result concerning multitype trees.

\begin{theorem}[Stephenson \cite{Ste18}, Theorem $3.1$]
\label{thm:loclimmulti}
Assume that $\bze$ is nondegenerate, critical and irreducible, and that $\bze$ has small exponential moments, in the sense that

\begin{align*}
\exists z > 1, \forall i \in [K], \sum_{w \in \cW_K} \zeta^{(i)}(w) z^{\sum w_i} < \infty.
\end{align*}

Fix in addition $\Gamma := (\gamma_1, \ldots, \gamma_K) \in \cM_{1,K}(\N)$, such that at least one of the $\gamma_i$'s is nonzero. Fix $j \in [K]$, and let $(k_n)_{n \geq 1}$ be a sequence of positive integers going to $+\infty$, such that, for all $n$
\begin{align*}
\P \left( \sum_{i=1}^d \gamma_i N_i\left( \cT^{(j)} \right) = k_n \right) > 0.
\end{align*}

Then, we have

\begin{align*}
\cT_{\Gamma,k_n}^{(j)} \underset{n \rightarrow \infty}{\overset{(d)}{\rightarrow}} \cT_{*}^{(j)},
\end{align*}
where $\cT_{*}^{(j)}$ is the multitype Kesten tree associated to $\tbze$. 
\end{theorem}

Corollary \ref{cor:loclim} is now just a consequence of Theorems \ref{thm:main} and \ref{thm:loclimmulti}.

\begin{proof}[Proof of Corollary \ref{cor:loclim}]
Let us consider a critical distribution $\tbze$ which is $\Gamma$-equivalent to $\bze$. Such a distribution exists by Theorem \ref{thm:main}. It is clear that, since $\bze$ is entire, nondegenerate and irreducible, the same holds for $\tbze$. In particular, since it is entire it has small exponential moments. The result follows.
\end{proof}

We finally prove the second part of Theorem \ref{thm:main}.

\begin{proof}[End of the proof of Theorem \ref{thm:main}]
Observe that, for any $j \in [K]$, the distribution of the tree $\cT_*^{(j)}$ of Theorem \ref{thm:loclimmulti} uniquely determines $\tbze$. The uniqueness of a critical distribution that is $\Gamma$-equivalent to $\bze$ then follows directly.
\end{proof}

\section{Open questions}
\label{sec:questions}

Here are some related open questions, mainly about the assumptions that we make on $\bze$.

\begin{enumerate}[label=(\textbf{Q.\arabic*})]
\item Is it possible to loosen \ref{condition}, to allow matrices $\Gamma$ with coefficients equal to $0$? This would allow us to use Miermont \cite{Mie08}, Haas-Stephenson \cite{HS21}, Stephenson \cite{Ste18}, to obtain for free new limiting results for noncritical trees. Furthermore, \ref{condition} is only used at one point in the proof, to prove Theorem \ref{thm:aroundzero}.
\item To our knowledge, no scaling limit result exists when $\Gamma$ is not $(1,0,\ldots,0)$. Such results for critical trees would imply, by Theorem \ref{thm:main}, the same result for a larger class of trees.
\item As in the monotype case, it is possible to loosen Assumption \ref{cond:entire}. However, this would lead to new technical difficulties that we prefer not to tackle in this paper.
\item It would be interesting to loosen \ref{cond:jamaiszero}, and allow some types to always have children. This assumption is not relevant in the critical case, and should not be in our case either. However it is central in the proof of Theorem \ref{thm:main}, and it does not seem clear how to get rid of it.
\item The strongest assumption made in this paper is Assumption \ref{cond:escape}, about the behaviour of $\tilde{\rho}$ on $\cC$ far from the origin. Although such an estimate seems to be mandatory in our case, the result of Theorem \ref{thm:main} seems to hold even without this assumption. Proving it would however require a different argument.
\item Following Remark \ref{rk:equiv}, in the definition of $\Gamma$-equivalence for distributions, does (ii) imply (i) if $\Gamma \in \cM_{1,K}(\Z_+)$?
\item We can study the set $A_+$ directly in the case $rk(\Gamma) \geq 2$. For instance, if $rk(\Gamma)=K$, we have $A=\R^K$. Is $A_+$ a $rk(\Gamma)$-dimensional manifold with boundary? What can be said about it?
\item Is Corollary \ref{cor:loclim} still true for $rk(\Gamma) \geq 2$?
\end{enumerate}

\bibliographystyle{abbrv}
\bibliography{BibTreeInv}

\end{document}